\documentclass[11pt]{amsart}
\usepackage{graphicx, color}
\usepackage{amscd}
\usepackage{amsmath,empheq}
\usepackage{amsfonts}
\usepackage{amssymb}
\usepackage{mathrsfs}
\usepackage[all]{xy}
\newtheorem{theorem}{Theorem}[section]

\newtheorem{lemma}[theorem]{Lemma}
\newtheorem{prop}[theorem]{Proposition}
\newtheorem{remark}{Remark}[section]
\newtheorem{corollary}[theorem]{Corollary}

\newenvironment{proof-sketch}{\noindent{\bf Sketch of Proof}\hspace*{1em}}{\qed\bigskip}

\everymath{\displaystyle}
\newcommand{\RR}{\mathbb R}
\newcommand{\NN}{\mathbb N}

\renewcommand{\leq}{\leqslant}

\renewcommand{\geq}{\geqslant}
\begin{document}
\title[Analysis of a Robin quasilinear problem]{Positive solutions for the Robin $p$-Laplacian\\ plus an indefinite potential}
\author[N.S. Papageorgiou]{Nikolaos S. Papageorgiou}
\address[N.S. Papageorgiou]{National Technical University, Department of Mathematics,
				Zografou Campus, Athens 15780, Greece \& Institute of Mathematics, Physics and Mechanics, 1000 Ljubljana, Slovenia}
\email{\tt npapg@math.ntua.gr}
\author[V.D. R\u{a}dulescu]{Vicen\c{t}iu D. R\u{a}dulescu}
\address[V.D. R\u{a}dulescu]{Institute of Mathematics, Physics and Mechanics, 1000 Ljubljana, Slovenia \& Faculty of Applied Mathematics, AGH University of Science and Technology, 30-059 Krak\'ow, Poland \& Institute of Mathematics ``Simion Stoilow" of the Romanian Academy, P.O. Box 1-764, 014700 Bucharest, Romania}
\email{\tt vicentiu.radulescu@imfm.si}
\author[D.D. Repov\v{s}]{Du\v{s}an D. Repov\v{s}}
\address[D.D. Repov\v{s}]{Faculty of Education and Faculty of Mathematics and Physics, University of Ljubljana \& Institute of Mathematics, Physics and Mechanics, 1000 Ljubljana, Slovenia}
\email{\tt dusan.repovs@guest.arnes.si}
\begin{abstract}
We consider a nonlinear elliptic equation driven by the Robin $p$-Laplacian plus an indefinite potential. In the reaction we have the competing effects of a strictly $(p-1)$-sublinear parametric term and of a $(p-1)$-linear and nonuniformly nonresonant term. We study the set of positive solutions as the parameter $\lambda>0$ varies. We prove a bifurcation-type result for large values of the positive parameter $\lambda$. Also, we show that for all admissible $\lambda>0$, the problem has a smallest positive solution $\overline{u}_\lambda$ and we study the monotonicity and continuity properties of the map $\lambda\mapsto\overline{u}_\lambda$.
\end{abstract}
\maketitle

\section{Introduction}

Let $\Omega\subseteq\RR^N$ be a bounded domain with a $C^2$-boundary $\partial\Omega$. In this paper we study the following nonlinear parametric Robin problem:
\begin{equation}
	\left\{
		\begin{array}{ll}
			-\Delta_p u(z) + \xi(z)u(z)^{p-1} = f(z,u(z),\lambda) + g(z,u(z))\ \mbox{in}\ \Omega, \\
			\frac{\partial u}{\partial n_p} + \beta(z)u(z)^{p-1}=0\ \mbox{on}\ \partial\Omega,\quad u>0\ \mbox{in}\ \Omega,\ \lambda>0.
		\end{array}
	\right\}
	\tag{$P_\lambda$}
	\label{eqp}
\end{equation}

In this problem, $\Delta_p$ denotes the p-Laplace differential operator defined by
$$
\Delta_p u={\rm div}\,(|Du|^{p-2}Du)\ \mbox{for all}\ u\in W^{1,p}(\Omega),\ 1<p<\infty.
$$

The potential function $\xi\in L^\infty(\Omega)$ is in general indefinite (that is, sign-changing). Therefore the differential operator (the left-hand side of \eqref{eqp}) need not be  coercive. In the reaction (the right-hand side of \eqref{eqp}), we have the competing effects of two terms.
The first is a parametric function which is strictly $(p-1)$-sublinear near $+\infty$. The second function (the perturbation of the parametric term), is $(p-1)$-linear near $+\infty$. Both functions are Carath\'eodory (that is, for all $x\in\RR$ the mappings $z\mapsto f(z,x,\lambda)$ and $x\mapsto g(z,x)$ are measurable and for all $z\in\Omega$ the functions $x\mapsto f(z,x,\lambda)$ and $x\mapsto g(z,x)$ are continuous). 

In the boundary condition, $\frac{\partial u}{\partial n_p}$ denotes the conormal derivative of $u$, defined by extension of the map
$$
C^1(\overline\Omega)\ni u\mapsto |Du|^{p-2}(Du,n)_{\RR^N} = |Du|^{p-2}\frac{\partial u}{\partial n},
$$
with $n(\cdot)$ being the outward unit normal on $\partial\Omega$. 
This map is uniformly continuous from $C^1(\overline\Omega)$ into $L^p(\partial\Omega)$ (in fact, it is locally Lipschitz if $p\geq2$ and H\"older continuous if $1<p<2$). Also, $C^1(\overline\Omega)$ is dense in $W^{1,p}(\Omega)$. So, this map admits a unique extension to the whole Sobolev space.

Our aim in this paper is to study the nonexistence, existence and multiplicity of positive solutions for problem \eqref{eqp} as the parameter $\lambda$ moves along the positive semiaxis $(0,+\infty)$.
We prove a bifurcation-type result for large values of the parameter (cf. Theorem~\ref{th14}). More precisely, we show that there is a critical parameter value $\lambda^*>0$ such that:\\
(i) for all $\lambda>\lambda^*$, problem \eqref{eqp} has at least two positive solutions;\\
(ii) for all $\lambda=\lambda^*$, problem \eqref{eqp} has at least one positive solution;\\
 (iii) for all $0<\lambda<\lambda^*$, problem \eqref{eqp} has no positive solutions.

Moreover, we show that for every admissible parameter $\lambda\in[\lambda^*,+\infty)$, problem \eqref{eqp} has a smallest positive solution $\overline{u}_\lambda$ and we examine the continuity and monotocicity properties of the map $\lambda\mapsto\overline{u}_\lambda$.

The first such bifurcation-type result for parametric elliptic equations with competing nonlinearities was proved by Ambrosetti, Brezis \& Cerami \cite{2} (semilinear Dirichlet problems with concave-convex reaction). Their work was extended to Dirichlet $p$-Laplace equations by Garcia Azorero, Manfredi \& Peral Alonso \cite{6}, Guo \& Zhang \cite{9}, Hu \& Papageorgiou \cite{10}. For equations of logistic type there are the works of R\u{a}dulescu \& Repov\v{s} \cite{19} (semilinear Dirichlet problems) and Cardinali, Papageorgiou \& Rubbioni \cite{4} (nonlinear Neumann problems).

 For Robin problems, we mention the work of Papageorgiou \& R\u{a}dulescu \cite{15}. In all aforementioned works the differential operator is coercive and the reaction has a different pair of competing nonlinearities. In the present paper we distinguish a new class of competition phenomena, which lead to bifurcation-type results. In fact, the behaviour of the set of positive solutions as the parameter $\lambda>0$ varies, is similar to that of superdiffusive logistic equations, since the ``bifurcation" occurs for large values of $\lambda>0$.

Our method of proof uses variational tools from critical point theory together with suitable truncation, perturbation and comparison arguments.

\section{Mathematical background and hypotheses}

Suppose that $X$ is a Banach space. We denote by $X^*$  the topological dual of $X$ and by $\langle\cdot,\cdot\rangle$ the duality brackets for the pair $(X^*,X)$.

Given $\varphi\in C^1(X,\RR)$ we say that $\varphi$ satisfies the ``Palais-Smale condition" (the ``PS-condition" for short) if the following property holds:
\begin{equation*}
	\begin{array}{cc}
		\mbox{``Every sequence}\ \{u_n\}_{n\geq1}\subseteq X \mbox{such that}\\
		\{\varphi(u_n)\}_{n\geq1}\subseteq\RR\ \mbox{is bounded and}\ \varphi'(u_n)\rightarrow 0\ \mbox{in}\ X^*\ \mbox{as}\ n\rightarrow\infty,  \\
		\mbox{admits a strongly convergent subsequence".}
	\end{array}
\end{equation*}

This is a compactness-type condition on the functional $\varphi$. Using this condition, one can prove a deformation theorem from which follows the minimax theory for the critical values of $\varphi$. Prominent in this theory is the so-called ``mountain pass theorem", which we recall here because we will use it in the sequel.

\begin{theorem}\label{th1}
		Assume that $\varphi\in C^1(X,\RR)$ satisfies the PS-condition, $u_0,\,u_1\in X$, $||u_1-u_0||>\rho>0$,
	$$\max\{\varphi(u_0),\varphi(u_1)\}<\inf\{\varphi(u):||u-u_0||=\rho\}=m_{\rho}$$
	and $c=\inf\limits_{\gamma\in\Gamma}\max\limits_{0\leq t\leq 1}\ \varphi(\gamma(t))$, where $$\Gamma=\{\gamma\in C([0,1],X):\gamma(0)=u_0,\gamma(1)=u_1\}.$$ Then $c\geq m_{\rho}$ and $c$ is a critical value of $\varphi$ (that is, we can find $\hat{u}\in X$ such that $\varphi'(\hat{u})=0$ and $\varphi(\hat{u})=c$).
\end{theorem}

\begin{remark}
	We note that if $\varphi'=A+K$, with $A:X\rightarrow X^*$ a continuous map of type $(S)_+$ (that is, if $u_n\xrightarrow{w}u$ in $X$ and $\limsup_{n\rightarrow\infty}\langle A(u_n),u_n-u\rangle\leq0$, then $u_n\rightarrow u$ in $X$), and $K:X\rightarrow X^*$ is completely continuous (that is, if $u_n\xrightarrow{w}u$ in $X$, then $K(u_n)\rightarrow K(u)$ in $X^*$), then $\varphi$ satisfies the PS-condition (see Marano \& Papageorgiou \cite[Proposition 2.2]{13}). This is the case in our setting.
\end{remark}

The analysis of problem \eqref{eqp} involves the Sobolev space $W^{1,p}(\Omega)$, the Banach space $C^1(\overline\Omega)$, and the ``boundary" Lebesgue space $L^p(\partial\Omega)$.

We denote by $||\cdot||$ the norm of the Sobolev space $W^{1,p}(\Omega)$ defined by
$$
||u||=\left(||u||^p_p + ||Du||^p_p\right)^\frac{1}{p}\ \mbox{for all}\ u\in W^{1,p}(\Omega).
$$

The space $C^1(\overline\Omega)$ is an ordered Banach space with positive (order) cone
$$
C_+=\{u\in C^1(\overline\Omega):u(z)\geq0\ \mbox{for all}\ z\in\overline\Omega\}.
$$

This cone has a nonempty interior given by
$$
{\rm int}\,C_+=\left\{u\in C_+: u(z)>0\ \mbox{for all}\ z\in\Omega,\ \frac{\partial u}{\partial n}|_{\partial\Omega\cap u^{-1}(0)}<0\right\}.
$$

We will also use the set $D_+\subseteq C_+$ defined by
$$
D_+ = \{u\in C_+:u(z)>0\ \mbox{for all}\ z\in\overline\Omega\}.
$$

Evidently, $D_+$ is open in $C^1(\overline\Omega)$ and $D_+\subseteq {\rm int}\, C_+$. In fact, $D_+$ is the interior of $C_+$ when $C^1(\overline\Omega)$ is furnished with the coarser relative $C(\overline\Omega)$-norm topology.

On $\partial\Omega$ we introduce the $(N-1)$-dimensional Hausdorff (surface) measure $\sigma(\cdot)$. Using $\sigma(\cdot)$, we can define in the usual way the boundary Lebesgue spaces $L^q(\partial\Omega), 1\leq q\leq\infty$. From the theory of Sobolev spaces we know that there exists a unique continuous linear map $\gamma_0:W^{1,p}(\Omega)\rightarrow L^p(\partial\Omega)$, known as the ``trace map", such that
$$
\gamma_0(u)=u|_{\partial\Omega}\ \mbox{for all}\ u\in W^{1,p}(\Omega)\cap C(\overline\Omega).
$$

So, the trace map gives a meaning to the notion of ``boundary values" for any Sobolev function. The trace map is not surjective (in fact, ${\rm im}\,\gamma_0=W^{\frac{1}{p'},p}(\partial\Omega)$, with $\frac{1}{p}+\frac{1}{p'}=1$) and ${\rm ker}\,\gamma_0=W^{1,p}_0(\Omega)$. Moreover, $\gamma_0$ is a compact map into $L^q(\partial\Omega)$ for all $q\in[1,\frac{(N-1)p}{N-p})$ if $p<N$ and into $L^p(\partial\Omega)$ for all $1\leq q<\infty$ if $N\leq p$. In the sequel, for the sake of notational simplicity, we will drop the use of the trace map $\gamma_0$. All restrictions of Sobolev functions on $\partial\Omega$ are understood in the sense of traces.

Let $A:W^{1,p}(\Omega)\rightarrow W^{1,p}(\Omega)^*$ be the nonlinear map defined by
$$
\langle A(u),h\rangle = \int_\Omega|Du|^{p-2}(Du,Dh)_{\RR^N}dz\ \mbox{for all}\ u,h\in W^{1,p}(\Omega).
$$

In the next proposition, we have collected the main properties of this map (see Gasinski \& Papageorgiou \cite[p. 279]{8}).

\begin{prop}\label{prop2}
	The map $A(\cdot)$ is bounded (that is, it maps bounded sets to bounded sets), continuous, monotone (thus, maximal monotone, too) and of type $(S)_+$.\end{prop}

 Now we introduce our conditions on the potential function $\xi(\cdot)$ and on the boundary coefficient $\beta(\cdot)$:

\smallskip
$H(\xi):$ $\xi\in L^\infty(\Omega)$

\smallskip
$H(\beta):$ $\beta\in C^{0,\alpha}(\partial\Omega)$ for some $0<\alpha<1$ and $\beta(z)\geq0$ for all $z\in\partial\Omega$.

\begin{remark}
	When $\beta\equiv0$, we have the Neumann problem.
\end{remark}

Let $\gamma_p: W^{1,p}(\Omega)\rightarrow\RR$ be the $C^1$-functional defined by
$$
\gamma_p(u)=||Du||^p_p + \int_\Omega\xi(z)|u|^pdz + \int_{\partial\Omega}\beta(z)|u|^pd\sigma\ \mbox{for all}\ u\in W^{1,p}(\Omega).
$$

Also, let $f_0:\Omega\times\RR\rightarrow\RR$ be a Carath\'eodory function that satisfies
$$
|f_0(z,x)|\leq a(z)(1+|x|^{r-1})\ \mbox{for almost all}\ z\in\Omega\ \mbox{and all}\ x\in\RR,
$$
with $a_0\in L^\infty(\Omega), 1<r\leq p^*=\left\{\begin{array}{ll}
\frac{Np}{N-p} & \mbox{if}\ p<N\\
+\infty & \mbox{if}\ N\leq p
\end{array}\right.$ (the critical Sobolev exponent).

 We set $F_0(z,x)=\int^x_0f_0(z,s)ds$ and consider the $C^1$-functional $\varphi_0:W^{1,p}(\Omega)\rightarrow\RR$ defined by
$$
\varphi_0(u) = \frac{1}{p}\gamma_p(u) - \int_\Omega F_0(z,u)dz\ \mbox{for all}\ u\in W^{1,p}(\Omega).
$$

In the framework of variational methods, the local minimizers of $\varphi_0$ play an important role. As we will see in the sequel, solutions of the problem are often generated by minimizing $\varphi_0$ on a constrained set defined by using the usual pointwise order on $W^{1,p}(\Omega)$ (this is done via truncation of $f_0(z,\cdot)$). It is well-known that the order cone $$W_+=\{u\in W^{1,p}(\Omega):u(z)\geq0\ \mbox{for almost all}\ z\in\Omega\}$$ of $W^{1,p}(\Omega)$ has an empty interior. So, it is not clear if the constrained minimizer is in fact,
 an unconstrained local minimizer of $\varphi_0$ over all of $W^{1,p}(\Omega)$.

The next result is helpful in this direction. It is a special case of a more general result that can be found in Papageorgiou \& R\u{a}dulescu \cite{16}. The first to prove this relation between H\"{o}lder and Sobolev local minimizers were Brezis \& Nirenberg \cite{3}.

\begin{prop}\label{prop3}
Assume that $u_0\in W^{1,p}(\Omega)$ is a local $C^1(\overline\Omega)$-minimizer of $\varphi_0$, that is, there exists $\rho_0>0$ such that
	$$
	\varphi_0(u_0)\leq \varphi_0(u_0+h)\ \mbox{for all}\ h\in C^1(\overline\Omega)\ \mbox{with}\ ||h||_{C^1(\overline\Omega)}\leq\rho_0.
	$$
	Then $u_0\in C^{1,\vartheta}(\overline\Omega)$ with $\vartheta\in(0,1)$ and $u_0$ is also a local $W^{1,p}(\Omega)$-minimizer of $\varphi_0$, that is, there exists $\rho_1>0$ such that
	$$
	\varphi_0(u_0)\leq \varphi_0(u_0+h)\ \mbox{for all}\ h\in W^{1,p}(\Omega)  \ \mbox{with}\ ||h||\leq\rho_1.
	$$
\end{prop}

As we have already mentioned in the first section of this paper, our approach involves also comparison arguments. The next proposition will be helpful in this direction. It is a special case of a more general result of Papageorgiou, R\u{a}dulescu \& Repov\v{s} \cite{18}.

\begin{prop}\label{prop4}
	Assume that $h_1,h_2,\vartheta\in L^\infty(\Omega), \vartheta(z)\geq0$ for almost all $z\in\Omega$,
	$$
	0<\eta\leq h_2(z)-h_1(z)\ \mbox{for almost all}\ z\in\Omega,
	$$
	and $u_1,u_2\in C^{1,\mu}(\overline\Omega)$ with $0<\mu\leq1$ are such that $u_1\leq u_2$ and
	$$
	\begin{array}{ll}
		-\Delta_p u_1 + \vartheta(z)|u_1|^{p-2}u_1 = h_1,\\
		-\Delta_p u_2 + \vartheta(z)|u_2|^{p-2}u_2 = h_2\ \mbox{for almost all}\ z\in\Omega.
	\end{array}
	$$
	Then $u_2-u_1\in {\rm int}\,C_+$.
\end{prop}

Next, we consider the following nonlinear eigenvalue problem
\begin{equation}\label{eq1}
	\left\{
		\begin{array}{ll}
			-\Delta_p u(z) + \xi(z)|u(z)|^{p-2}u(z) =\hat\lambda|u(z)|^{p-2}u(z)\ \mbox{in}\ \Omega, \\
			\frac{\partial u}{\partial n_p} + \beta(z)|u|^{p-2}u=0\ \mbox{on}\ \partial\Omega.
		\end{array}
	\right\}
\end{equation}

We say that $\hat\lambda\in\RR$ is an ``eigenvalue" if problem (\ref{eq1}) admits a nontrivial solution $\hat{u}$, which is known as an ``eigenfunction" corresponding to $\hat\lambda$. We denote by $\hat\sigma(p)$ the set of eigenvalues of problem (\ref{eq1}). It is easy to see that $\hat\sigma(p)\subseteq\RR$ is closed and has a smallest element $\hat\lambda_1=\hat\lambda_1(p,\xi,\beta)\in\RR$ (the first eigenvalue), which has the following properties (for details, we refer to Papageorgiou \& R\u{a}dulescu \cite{15} and Fragnelli, Mugnai \& Papageorgiou \cite{5}).

\begin{prop}\label{prop5}
	If hypotheses $H(\xi), H(\beta)$ are satisfied, then problem (\ref{eq1}) has a smallest eigenvalue $\hat\lambda_1\in\RR$ such that
	\begin{itemize}
  		\item [(a)] $\hat\lambda_1$ is isolated in $\hat\sigma(p)$ (that is, there exists $\epsilon>0$ such that $(\hat\lambda_1,\hat\lambda,+\epsilon)\cap\hat\sigma(p)=\emptyset$);
  		\item [(b)] $\hat\lambda_1$ is simple (that is, if $\hat{u},\hat{v}$ are eigenfunctions corresponding to $\hat\lambda_1$, then $\hat{u}=\eta\hat{v}$ for some $\eta\in\RR\backslash\{0\}$);
  		\begin{equation}\label{eq2}
  		(c)\hspace{3.15cm}\hat\lambda_1=\inf\left\{\frac{\gamma_0(u)}{||u||^p_p}:u\in W^{1,p}(\Omega),u\neq0\right\}.\hspace{4cm}
  		\end{equation}
	\end{itemize}
\end{prop}

\begin{remark}
	The infimum in (\ref{eq2}) is realized on the corresponding one-dimensional eigenspace.
\end{remark}

It follows from (\ref{eq2}) that the elements of this eigenspace have fixed sign. We denote by $\hat{u}_1$  the positive, $L^p$-normalized (that is, $||\hat{u}_1||_p=1$) eigenfunction corresponding to $\hat\lambda_1$.
We know that $\hat{u}_1\in D_+$ (see \cite{15}, \cite{5}). Also, every eigenvalue different from $\hat\lambda_1$ has eigenfunctions in $C^1(\overline\Omega)$ which are nodal (that is, sign-changing). Finally, if $\xi\in L^\infty(\Omega), \xi(z)\geq0$ for almost all $z\in\Omega$ and either $\xi\not\equiv0$ or $\beta\not\equiv0$, then $\hat\lambda_1>0$.

An easy consequence of the above properties is the following lemma (see Mugnai \& Papageorgiou \cite[Lemma 4.11]{14}).

\begin{lemma}\label{lem6}
	If hypotheses $H(\xi), H(\beta)$ hold, $\eta\in L^\infty(\Omega), \eta(z)\leq\hat\lambda_1$ for almost all $z\in\Omega$, and the inequality is strict on a set of positive measure, then there exists $c_0>0$ such that
	$$
	c_0||u||^p\leq\gamma_p(u)-\int_\Omega\eta(z)|u|^pdz\ \mbox{for all}\ u\in W^{1,p}(\Omega).
	$$
\end{lemma}

The hypotheses on the two terms of the reaction of \eqref{eqp} are the following.

\smallskip
$H(f)$ $f:\Omega\times\RR\times(0,+\infty)\rightarrow\RR$ is a Carath\'eodory function such that for all $\lambda>0,\ f(z,x,\lambda)\geq0$ for almost all $z\in\Omega$ and all $x\geq0$, $f(z,0,\lambda)=0$ for almost all $z\in\Omega$, and
	\begin{itemize}
		\item [(i)] for every $\rho>0$, $\lambda_0>0$, there exists $a_{\rho,\lambda_0}\in L^\infty(\Omega)$ such that
			$$
			0\leq f(z,x,\lambda) \leq a_{\rho,\lambda_0}(z)\ \mbox{for almost all}\ z\in\Omega\ \mbox{and all}\ 0\leq x\leq \rho, \ 0<\lambda\leq\lambda_0;
			$$
		\item [(ii)] for every $\lambda>0$, we have
			$$
			\lim_{x\rightarrow+\infty}\frac{f(z,x,\lambda)}{x^{p-1}}=\lim_{x\rightarrow0^+}\frac{f(z,x,\lambda)}{x^{p-1}}=0\ \mbox{uniformly for almost all}\ z\in\Omega;
			$$
		\item [(iii)] if $F(z,x,\lambda)=\int_0^x f(z,s,\lambda)ds$, then there exist $v_0\in L^p(\Omega)$ and $\tilde\lambda>0$ such that $\int_\Omega F(z,v_0(z),\lambda)dz>0$ for all $\lambda>\tilde{\lambda}$;
		\item [(iv)]
				$\bullet$ we have $f(z,x,\lambda)\rightarrow0^+$ as $\lambda\rightarrow0^+$ uniformly for almost all $z\in\Omega$ and all $x\in C\subseteq\RR_+$ bounded, $f(z,x,\lambda)\rightarrow+\infty$ as $\lambda\rightarrow+\infty$ for almost all $z\in\Omega$ and all $x>0$;\\
				$\bullet$ for every $s>0$, we can find $\tilde\eta_s>0$ such that
					$$
					0<\tilde\eta_s\leq f(z,x,\mu) - f(z,x,\lambda)\ \mbox{for almost all}\ z\in\Omega\  \mbox{and all}\ x\geq s,\ 0<\lambda<\mu.
					$$
	\end{itemize}

\begin{remark}
	Since we are looking for positive solutions and all the above hypotheses concern the positive semiaxis $\RR_+=[0,+\infty)$, we may assume without any loss of generality that
	\begin{equation}\label{eq3}
		f(z,\cdot,\lambda)|_{(-\infty,0]}=0\ \mbox{for almost all}\ z\in\Omega\ \mbox{and all}\ \lambda>0.
	\end{equation}
\end{remark}

Note that hypothesis $H(f)(ii)$ implies that $f(z,\cdot,\lambda)$ is strictly $(p-1)$-sublinear near $+\infty$ and also near $0^+$. Hypothesis $H(f)(iii)$ is satisfied if there exists $\tilde\lambda>0$ such that $L(z)=\{x\in\RR: f(z,x,\lambda)>0\}$ is nonempty for almost all $z\in\Omega$ and all $\lambda>\tilde\lambda$. Finally, note that  hypothesis $H(f)(iv)$ implies that for almost all $z\in\Omega$ and all $x>0$, the mapping $\lambda\mapsto f(z,x,\lambda)$ is strictly increasing.

\smallskip
$H(g)$: $g:\Omega\times\RR\rightarrow\RR$ is a Carath\'eodory function such that $g(z,0)=0$ for almost all $z\in\Omega$ and
\begin{itemize}
	\item [(i)] there exist $a\in L^\infty(\Omega)$ and $p\leq r<p^*$ such that
		$$
		(g(z,x))\leq a(1)(1+x^{r-1})\ \mbox{for almost all}\ z\in\Omega\ \mbox{and all}\ x\geq0;
		$$
	\item [(ii)] there exists a function $\eta_0\in L^\infty(\Omega)$ such that
		$$
		\begin{array}{ll}
			\eta_0(z)\leq\hat\lambda_1\ \mbox{for almost all}\ z\in\Omega,\ \eta_0\not\equiv\hat\lambda_1,\\
			\limsup_{x\rightarrow+\infty}\frac{g(z,x)}{x^{p-1}}\leq\eta_0(z)\ \mbox{and}\ \limsup_{x\rightarrow0^+}\frac{g(z,x)}{x^{p-1}}\leq\eta_0(z)\ \mbox{uniformly for almost all}\ z\in\Omega;
		\end{array}
		$$
	\item [(iii)] for almost all $z\in\Omega$, the mapping $x\mapsto\frac{g(z,x)}{x^{p-1}}$ is nondecreasing on $(0,+\infty)$.
\end{itemize}

\begin{remark}
	As we did for $f(z,\cdot,\lambda)$, without any loss of generality, we may assume that
	\begin{equation}\label{eq4}
		g(z,\cdot)|_{(-\infty,0]}=0\ \mbox{for almost all}\ z\in\Omega.
	\end{equation}
\end{remark}

Hypothesis $H(g)(ii)$ says that asymptotically at $+\infty$ and at $0^+$ we have nonuniform nonresonance with respect to $\hat\lambda_1$ from the left.

\smallskip
$H_0:$ for every $\rho>0$, $\tilde\lambda>0$, we can find $\hat\xi_0^{\tilde\lambda}>0$ such that for almost all $z\in\Omega$ and all $0<\lambda\leq\lambda_0$, the function $x\mapsto f(z,x,\lambda)+g(z,x)+\hat\xi_\rho^{\hat\lambda}x^{p-1}$ is nondecreasing on $[0,\rho]$.

\begin{remark}
	This hypothesis is satisfied if, for example, for almost all $z\in\Omega$ and every $\lambda>0$, the functions $f(z,\cdot,\lambda)$ and $g(z,\cdot)$ are differentiable and for every $\rho>0$, $\hat\lambda>0$, there exists $\hat\xi_\rho^{\tilde\lambda}>0$ such that
	$$
	(f'(z,x,\lambda)+g'_x(z,x))x\geq - \hat{\xi}_\rho^{\tilde\lambda}x^{p-1}\ \mbox{for almost all}\ z\in\Omega\ \mbox{and all}\ 0\leq x\leq\rho.
	$$
\end{remark}

\textit{Examples:} The following pairs of functions $f$ and $g$ satisfy hypotheses $H(f), H(g), H_0$. For the sake of simplicity we drop the $z$-dependence. Also recall (\ref{eq3}) and (\ref{eq4}).

\medskip\noindent
$
	f_1(x,\lambda) = \left\{\begin{array}{lll}
&\lambda x^{p-1}\ln (1+x) &\quad \mbox{if}\ 0\leq x\leq 1\\
&\lambda x^{q-1}&\quad\mbox{if}\ 1<x\end{array}\right.\qquad 1<q<p
$
	$$
\begin{array}{ll}
	g_1(x)=\eta x^{p-1}\ & \mbox{for}\ x\geq0,\ \eta<\hat{\lambda}_1, \\
	\\
	f_2(x,\lambda)=\left\{\begin{array}{ll}
		\lambda x^{r-1}\ &\mbox{if}\ 0\leq x\leq1\\
		\lambda x^{q-1}\ &\mbox{if}\ 1<x
	\end{array}\right.\ & 1<q<p<r, \\
	\\
	g_2(x)=\left\{\begin{array}{ll}
		cx^{\tau-1} - x^{q-1}\ &\mbox{if}\ 0\leq x\leq1\\
		\eta x^{p-1} + (c-1-\eta)\ &\mbox{if}\ 1<x
	\end{array}\right. & 1<q<p\leq\tau, \eta<\hat\lambda_1,\ c>\max\{\eta+1,0\},\\
	\\
	f_3(x,\lambda)=\left\{\begin{array}{ll}
		\lambda(x^{\tau-1}-x^{r-1})\ &\mbox{if}\ 0\leq x\leq1\\
		\lambda x^{q-1}\ln x\ &\mbox{if}\ 1<x
	\end{array}\right. & 1<q<p<\tau<r,\\
	\\
	g_3(x)=\left\{\begin{array}{ll}
		\eta(x^{p-1}+x^{r-1})\ &\mbox{if}\ 0\leq x\leq1\\
		\eta(x^{p-1}+x^{q-1})\ &\mbox{if}\ 1<x
	\end{array}\right. & 1<q<p<r,\ \eta<\hat\lambda_1, \\
	\\
	f_4(x,\lambda)=\left\{\begin{array}{ll}
		x^{\tau-1} & \mbox{if}\ 0\leq x\leq \rho(\lambda)\\
		x^{q-1} +\mu(\lambda)\ & \mbox{if}\ \rho(\lambda)<x
	\end{array}\right. \\
	\\
	g_4(x)=\eta x^{p-1}
\end{array}
$$
with $\rho:(0,+\infty)\rightarrow(0,+\infty)$ strictly increasing, continuous, $\rho(\lambda)\rightarrow0^+$ as $\lambda\rightarrow0^+$, $\rho(\lambda)\rightarrow+\infty$ as $\lambda\rightarrow+\infty$, $ \mu(\lambda)=[\rho(\lambda)^{\tau-1}-1]\rho(\lambda)^{q-1}$, $1<q<p<\tau$ and $\eta<\hat{\lambda}_1$.

Finally, we fix some basic notations which we will use throughout this paper. Let $x\in\RR$ and set $x^\pm=\max\{\pm x,0\}$. Then for $u\in W^{1,p}(\Omega)$ we define $u^\pm(\cdot)=u(\cdot)^\pm$. We know that
$$
u^\pm\in W^{1,p}(\Omega),\ u=u^+-u^-,\ |u|=u^++u^-.
$$

Also, if $u,\hat{u}\in W^{1,p}(\Omega)$ and $u\leq\hat{u}$, then
$$
[u,\hat{u}] = \{v\in W^{1,p}(\Omega): u(z)\leq v(z)\leq\hat{u}(z)\ \mbox{for almost all}\ z\in\Omega\}.
$$

We denote by ${\rm int}_{C^1(\overline\Omega)}[u,\hat{u}]$ the interior  of $[u,\hat{u}]\cap C^1(\overline\Omega)$ in $C^1(\overline\Omega)$.
Under the hypotheses on the data of problem $(P_\lambda)$, the main result of this paper is the following bifurcation-type theorem.

\begin{theorem}\label{th14}
 Assume that hypotheses $H(\xi),H(\beta),H(f),H(g),H_0$ hold. Then there exists $\lambda^*>0$ such that
	\begin{itemize}
		\item[(a)] for all $\lambda>\lambda^*$, problem \eqref{eqp} has at least two positive solutions
		$$u_0,\ \hat{u}\in D_+;$$
		\item[(b)] for $\lambda=\lambda^*$, problem \eqref{eqp} has at least one positive solution
		$$u_{\lambda^*}\in D_+;$$
		\item[(c)] for all $\lambda\in(0,\lambda^*)$, problem \eqref{eqp} has no positive solutions.
	\end{itemize}  
	\end{theorem}

Finally, if $\varphi\in C^1(X,\RR)$, then by $K_\varphi$ we denote the critical set of $\varphi$, that is,
$$
K_\varphi=\{u\in X:\varphi'(u)=0\}.
$$

\section{Positive solutions}
Throughout the rest of the paper we assume that hypotheses $H(\xi),\, H(\beta),\,
H(f),\, H(g),\, H_0$ are fulfilled.
We introduce the following  sets:
$$
\begin{array}{ll}
	\mathcal{L}=\{\lambda>0:\ \mbox{problem \eqref{eqp} admits a positive solution}\},\\
	S(\lambda)=\ \mbox{the set of positive solutions for problem \eqref{eqp}}.
\end{array}
$$

We set $\lambda^*=\inf\mathcal{L}$ with the usual convention that $\inf\emptyset=+\infty$.

\begin{prop}\label{prop8}
	We have $\mathcal{L}\neq\emptyset$ and so $0\leq\lambda^*<+\infty$.
\end{prop}

\begin{proof}
	From hypotheses $H(f)(i),(ii)$, we see that given $\epsilon>0$ and $\lambda>0$, we can find $c_1=c_1(\epsilon,\lambda)>0$ such that
	\begin{equation}\label{eq12}
		F(z,x,\lambda) \leq \frac{\epsilon}{p}x^p + c_1\ \mbox{for almost all}\ z\in\Omega\ \mbox{and all}\ x\geq0.
	\end{equation}
	
	Similarly, hypotheses $H(g)(i),(ii)$ imply that we can find $c_2=c_2(\epsilon)>0$ such that
	\begin{equation}\label{eq13}
		G(z,x) \leq (\eta_0(z) + \epsilon) x^p + c_2\ \mbox{for almost all}\ z\in\Omega\ \mbox{and all}\ x\geq0.
	\end{equation}
	
	Let $\mu>||\xi||_\infty$ (see hypothesis $H(\xi)$) and consider the Carath\'eodory function $k_\lambda(z,x)$ defined by
	$$
	k_\lambda(z,x) = f(z,x,\lambda) + g(z,x)\ \mbox{for all}\ (z,x)\in\Omega\times\RR, \ \lambda>0\ \mbox{(see (\ref{eq3}), (\ref{eq4}))}.
	$$
	
	We set $K_\lambda(z,x)=\int^x_0k_\lambda(z,s)ds$ and consider the $C^1$-functional $\Psi_\lambda: W^{1,p}(\Omega)\rightarrow\RR$ defined by
	$$
	\Psi_\lambda(u)=\frac{1}{p}\gamma_p(u) + \frac{\mu}{p}||u^-||^p_p - \int_\Omega K_\lambda(z,u)dz\ \mbox{for all}\ u\in W^{1,p}(\Omega).
	$$
	
	Using (\ref{eq12}) and (\ref{eq13}), we have for all $u\in W^{1,p}(\Omega)$.
$$	\begin{array}{ll}
	\displaystyle	\Psi_\lambda(u) &\displaystyle \geq c_3||u^-||^p + \frac{1}{p}\gamma_p(u^+) - \frac{1}{p}\int_\Omega(\eta_0(z)+2\epsilon)\,(u^+)^pdz - c_4\\
		& \displaystyle\mbox{for some}\ c_3,c_4>0\ \mbox{(recall that $\mu>||\xi||_\infty$)} \\
		& \displaystyle\geq c_3||u^-||^p + (c_0-2\epsilon)\,||u^+||^p - c_4.
	\end{array}$$
	
	Choosing $\epsilon\in(0,\frac{c_0}{2})$, we obtain
	\begin{eqnarray*}
		&& \Psi_\lambda(u) \geq c_5||u||^p - c_4\ \mbox{for some}\ c_5>0\ \mbox{and all}\ u\in W^{1,p}(\Omega),\\
		\Rightarrow && \Psi_\lambda(\cdot)\ \mbox{is coercive}.
	\end{eqnarray*}
	
	Also, using the Sobolev embedding theorem and the compactness of the trace map, we see that
	$$
	\Psi_\lambda(\cdot)\ \mbox{is sequentially weakly lower semicontinuous}.
	$$
	
	By the Weierstrass-Tonelli theorem, we can find $u_\lambda\in W^{1,p}(\Omega)$ such that
	\begin{equation}\label{eq14}
		\Psi_\lambda(u_\lambda) = \inf\left\{\Psi_\lambda(u): u\in W^{1,p}(\Omega)\right\}.
	\end{equation}
	
	Hypotheses $H(f)(i),(ii)$ imply that for every $\lambda>0$, we can find $c_6=c_6(\lambda)>0$ such that
	$$
	0\leq F(z,x,\lambda)\leq c_6 x^p\ \mbox{for almost all}\ z\in\Omega\ \mbox{and all}\ x\geq0.
	$$
	
	Evidently, in hypothesis $H(f)(iii)$ we can have $v_0\geq0$ (see (\ref{eq3})). Consider the continuous integral functional $i_\lambda: L^p(\Omega)\rightarrow\RR$ defined by
	\begin{eqnarray*}
		&& i_\lambda(v) = \int_\Omega F(z,v(z),\lambda)dz\ \mbox{for all}\ v\in L^p(\Omega),\\
		\Rightarrow && i_\lambda(v_0)>0\ \mbox{for all}\ \lambda>\tilde\lambda>0\ \mbox{(see hypothesis $H(f)(iii)$)}.
	\end{eqnarray*}
	
	Exploiting the density of $W^{1,p}(\Omega)$ in $L^p(\Omega)$, we can find $\tilde{v}_0\in W^{1,p}(\Omega)$, $\tilde{v}_0\geq0$, $\tilde{v}_0\neq0$, such that
	$$
	i_\lambda(\tilde{v}_0)>0\ \mbox{for all}\ \lambda>\tilde\lambda.
	$$
	
	Then using hypothesis $H(f)(iv)$ and Fatou's lemma, we can infer that
	\begin{equation}\label{eq15}
		\lim_{\lambda\rightarrow+\infty}\int_\Omega F(z,\tilde{v}_0,\lambda)dz=+\infty.
	\end{equation}
	
	On the other hand, hypothesis $H(g)(i)$ implies that if $G(z,x)=\int^x_0g(z,s)ds$, then
	\begin{equation}\label{eq16}
		\left|\int_\Omega G(z,\tilde{v}_0)dz\right|\leq c_7\ \mbox{for some}\ c_7>0.
	\end{equation}
	
	Now we see from (\ref{eq15}) and (\ref{eq16}) that for large enough $\lambda>\tilde\lambda$, we have
	\begin{eqnarray*}
		&& \Psi_\lambda(\tilde{v}_0)<0,\\
		\Rightarrow && \Psi_\lambda(u_\lambda)<0=\Psi_\lambda(0)\ \mbox{(see (\ref{eq14}))}\\
		\Rightarrow && u_\lambda\neq0.
	\end{eqnarray*}
	
	From (\ref{eq14}) we have
	\begin{eqnarray}
		&& \Psi'_\lambda(u_\lambda)=0, \nonumber \\
		\Rightarrow && \langle A(u_\lambda),h\rangle + \int_\Omega \xi(z)|u_\lambda|^{p-2}u_\lambda hd\sigma \int_{\partial\Omega}\beta(z)|u_\lambda|^{p-2}u_\lambda hd\sigma - \int_\Omega \mu(u^-_\lambda)^{p-1}hd\sigma \nonumber \\
		= && \int_\Omega[f(z,u_\lambda,\lambda) + g(z,u_\lambda)]hdz\ \mbox{for all}\ h\in W^{1,p}(\Omega). \label{eq17}
	\end{eqnarray}
	
	In (\ref{eq17}) we choose $h=-u^-_\lambda\in W^{1,p}(\Omega)$. Then
	\begin{eqnarray*}
		&& \gamma_p(u^-_\lambda) + \mu||u^-_\lambda||^p_p = 0\ \mbox{(see (\ref{eq3}), (\ref{eq4}))}, \\
		\Rightarrow && c_8||u^-_\lambda||^p\leq0\ \mbox{for some}\ c_8>0\ \mbox{(recall that $\mu>||\xi||_\infty$)}, \\
		\Rightarrow && u_\lambda\geq0, u_\lambda\neq0.
	\end{eqnarray*}
	
	Then it follows from (\ref{eq17}) that $u_\lambda\in S_\lambda\subseteq D_+$ and so for large enough $\lambda>\tilde\lambda$, we have $\lambda\in\mathcal{L}$, hence $\mathcal{L}\neq\emptyset$.
\end{proof}

\begin{prop}\label{prop7}
	For every $\lambda\in\mathcal{L}$ we have $S(\lambda)\subseteq D_+$ and $\lambda^*>0$.
\end{prop}

\begin{proof}
	Let $\lambda\in\mathcal{L}$ and let $u\in S(\lambda)$.
Reasoning as in Papageorgiou \& R\u{a}dulescu \cite{15}, using the nonlinear Green identity, we have
	\begin{equation}\label{eq5}
		\left\{
			\begin{array}{ll}
				-\Delta_p u(z) + \xi(z)u(z)^{p-1} = f(z,u(z),\lambda) + g(z,u(z))\ \mbox{for almost all}\ z\in\Omega,\\
				\frac{\partial u}{\partial n_p} + \beta(z)u^{p-1}=0\ \mbox{on}\ \partial\Omega.
			\end{array}
		\right\}
	\end{equation}
	
  By (\ref{eq5}) and Papageorgiou \& R\u{a}dulescu \cite{16} (see Proposition 7) we have
    $$
    u\in L^\infty(\Omega).
    $$

    Invoking Theorem 2 of Lieberman \cite{12}, we can infer that
    $$
    u\in C_+\backslash\{0\}.
    $$

    Let $\rho=||u||_\infty$ and let $\hat{\xi}^\lambda_\rho>0$ be as postulated by hypothesis $H_0$. Then
    \begin{equation}\label{eq6}
        \Delta_pu(z)\leq\left(||\xi||_\infty+\hat\xi^\lambda_\rho\right)u(z)^{p-1}\ \mbox{for almost all}\ z\in\Omega.
    \end{equation}

    From (\ref{eq6}) and  the nonlinear maximum principle (see, for example, Gasinski \& Papageorgiou \cite[p. 738]{7}), we have
    $$
    \begin{array}{ll}
        & u\in D_+,\\
        \Rightarrow & S(\lambda)\subseteq D_+\ \mbox{for all}\ \lambda>0.
    \end{array}
    $$

    Next, we show that $\lambda^*=\inf\mathcal{L}>0$. Hypotheses $H(f)(i),(ii),(iv)$ imply that given $\epsilon>0$, we can find $\overline\lambda>0$ such that
    \begin{equation}\label{eq7}
        0\leq f(z,x,\overline\lambda)\leq\epsilon x^{p-1}\ \mbox{for almost all}\ z\in\Omega\ \mbox{and all}\ x\geq0.
    \end{equation}

    Hypothesis $H(g)(ii)$ implies that we can find $M,\delta>0$ such that
    \begin{equation}\label{eq8}
        g(z,x)\leq(\eta_0(z)+\epsilon)x^{p-1}\ \mbox{for almost all}\ z\in\Omega\ \mbox{and all}\ x\geq M,\ 0\leq x\leq\delta.
    \end{equation}

    On the other hand, by hypothesis $H(g)(iii)$, we have for almost all $z\in\Omega$ and all $\delta\leq x\leq M$
	\begin{eqnarray}
		& \frac{g(z,x)}{x^{p-1}} & \leq \frac{g(z,M)}{M^{p-1}},\nonumber \\
		\Rightarrow & g(z,x) & \leq \frac{g(z,M)}{M^{p-1}}x^{p-1}\nonumber \\
		& & \leq (\eta_0(z)+\epsilon)x^{p-1}\ \mbox{(see (\ref{eq8}))}. \label{eq9}
	\end{eqnarray}

	So, by (\ref{eq8}) and (\ref{eq9}), we infer that
	\begin{equation}\label{eq10}
		g(z,x) \leq (\eta_0(z)+\epsilon)x^{p-1}\ \mbox{for almost all}\ z\in\Omega\ \mbox{and all}\ x\geq0.
	\end{equation}

	Let $\lambda\in(0,\overline\lambda)$ (see (\ref{eq7})) and assume that $\lambda\in\mathcal{L}$. Then from the first part of the proof, we know that we can find $u_\lambda\in S(\lambda)\subseteq D_+$. For every $h\in W^{1,p}(\Omega)$, $h\geq0$ we have
	\begin{eqnarray}
		&& \langle A(u_\lambda),h \rangle + \int_\Omega\xi(z)u^{p-1}_\lambda hdz + \int_{\partial\Omega} \beta(z)u^{p-1}_\lambda hd\sigma \nonumber \\
		& = & \int_\Omega[f(z,u_\lambda,\lambda) + g(z,u_\lambda)] hdz \nonumber \\
		& \leq & \int_\Omega(\eta_0(z) + 2\epsilon)u^{p-1}_\lambda hdz\ \mbox{(see (\ref{eq7}), (\ref{eq10}) and hypothesis $H(f)(iv)$)}. \label{eq11}
	\end{eqnarray}

	In (\ref{eq11}) we choose $h=u_\lambda\in W^{1,p}(\Omega), u_\lambda\geq0$. Then
	$$
	\begin{array}{ll}
		& \gamma_p(u_\lambda) - \int_\Omega\eta_0(z)u^{p-1}_\lambda dz \leq 2\epsilon||u_\lambda||^p, \\
		\Rightarrow & c_0 \leq 2\epsilon\ \mbox{(see Lemma \ref{lem6})},
	\end{array}
	$$

	Choosing $\epsilon\in(0,\frac{c_0}{2})$, we get a contradiction. Therefore $\lambda\not\in\mathcal{L}$ and so
	$$
	0<\overline\lambda\leq\lambda^*.
	$$
The proof is now complete.
\end{proof}

Next, we show that $\mathcal{L}$ is a half-line.

\begin{prop}\label{prop9}
	Assume that $\lambda\in\mathcal{L}$. Then $[\lambda,+\infty)\subseteq\mathcal{L}$.
\end{prop}
\begin{proof}
	Since $\lambda\in\mathcal{L}$, we can find $u_{\lambda}\in S(\lambda)\subseteq D_+$ (see Proposition \ref{prop7}). Let $\vartheta>\lambda$ and consider the following truncation-perturbation of the reaction in problem $(P_{\vartheta})$:
	\begin{eqnarray}\label{eq18}
		&&\hat{k}_{\vartheta}(z,x)=\left\{\begin{array}{ll}\
			f(z,u_{\lambda}(z),\vartheta)+g(z,u_{\lambda}(z))+\mu u_{\lambda}(z)^{p-1}&\mbox{if}\ x\leq u_{\lambda}(z)\\
			f(z,x,\vartheta)+g(z,x)+\mu x^{p-1}&\mbox{if}\ u_{\lambda}(z)<x.
		\end{array}\right.
	\end{eqnarray}
	
	Recall that $\mu>||\xi||_{\infty}$. We set $\hat{K}_{\vartheta}(z,x)=\int^x_0\hat{k}_{\vartheta}(z,s)ds$ and consider the $C^1$-functional $\hat{\psi}_{\vartheta}:W^{1,p}(\Omega)\rightarrow\RR$ defined by
	$$\hat{\psi}_{\vartheta}(u)=\frac{1}{p}\gamma_p(u)+\frac{\mu}{p}||u||^p_p-\int_{\Omega}\hat{K}_{\vartheta}(z,u)dz\ \mbox{for all}\ u\in W^{1,p}(\Omega).$$
	
	Reasoning as in the proof of Proposition \ref{prop8}, we can show that
	\begin{itemize}
		\item	$\hat{\psi}_{\vartheta}(\cdot)$ is coercive;
		\item $\hat{\psi}_{\vartheta}(\cdot)$ is sequentially weakly lower semicontinuous.
	\end{itemize}
	
	So, we can find $u_{\vartheta}\in W^{1,p}(\Omega)$ such that
	\begin{eqnarray*}
		&&\hat{\psi}_{\vartheta}(u_{\vartheta})=\inf\left\{\hat{\psi}_{\vartheta}(u):u\in W^{1,p}(\Omega)\right\},\\
		&\Rightarrow&\hat{\psi}'_{\vartheta}(u_{\vartheta})=0,
	\end{eqnarray*}
	\begin{eqnarray}\label{eq19}
		&\Rightarrow&\left\langle A(u_{\vartheta}),h\right\rangle+\int_{\Omega}(\xi(z)+\mu)|u_{\vartheta}|^{p-2}u_{\vartheta}hdz+\int_{\partial\Omega}\beta(z)|u_{\vartheta}|^{p-2}u_{\vartheta}hd\sigma=\nonumber\\
		&&\hspace{1cm}\int_{\Omega}\hat{k}_{\vartheta}(z,u_{\vartheta})hdz\ \mbox{for all}\ W^{1,p}(\Omega).
	\end{eqnarray}
	
	In (\ref{eq19}) we choose $h=(u_{\lambda}-u_{\vartheta})^+\in W^{1,p}(\Omega)$. Then we have
	\begin{eqnarray*}
		&&\left\langle A(u_{\vartheta}),(u_\lambda-u_{\vartheta})^+\right\rangle+\int_{\Omega}(\xi(z)+\mu)|u_{\vartheta}|^{p-2}u_{\vartheta}(u_\lambda-u_{\vartheta})^+dz+\\
		&&\hspace{1cm}\int_{\partial\Omega}\beta(z)|u_{\vartheta}|^{p-2}u_{\vartheta}(u_\lambda-u_{\vartheta})^+d\sigma\\
		&=&\int_{\Omega}[f(z,u_\lambda,\vartheta)+g(z,u_\lambda)+\mu u_\lambda^{p-1}](u_\lambda-u_{\vartheta})^+dz\ (\mbox{see (\ref{eq18})})\\
		&\geq&\int_\Omega[f(z,u_\lambda,\lambda)+g(z,u_\lambda)+\mu_\lambda^{p-1}](u_\lambda-u_{\vartheta})^+dz\ (\mbox{since}\ \lambda<\vartheta,\\
		&&\mbox{see hypothesis}\ H(f)(iv))\\
		&=&\left\langle A(u_\lambda),(u_\lambda-u_{\vartheta})^+\right\rangle+
\int_\Omega(\xi(z)+\mu)u_\lambda^{p-1}(u_\lambda-u_{\vartheta})^+dz+\int_{\partial\Omega}\beta(z)u_\lambda^{p-1}(u_\lambda-u_{\vartheta})^+d\sigma\\
		&&(\mbox{since}\ u_\lambda\in S(\lambda)),\\
		&\Rightarrow&u_\lambda\leq u_{\vartheta}\ (\mbox{see Proposition \ref{prop2} and recall that}\ \mu>||\xi||_{\infty}).
	\end{eqnarray*}
	
	Now equation (\ref{eq19}) becomes
	\begin{eqnarray*}
		&&\left\langle A(u_{\vartheta}),h\right\rangle+\int_\Omega\xi(z)u_{\vartheta}^{p-1}hdz+\int_{\partial\Omega}\beta(z)u_{\vartheta}^{p-1}hd\sigma=\int_\Omega[f(z,u_{\vartheta},\vartheta)+g(z,u_{\vartheta})]hdz\\
		&&\mbox{for all}\ h\in W^{1,p}(\Omega),\\
		&\Rightarrow&u_{\vartheta}\in S(\vartheta)\subseteq D_+\ \mbox{and so}\ \vartheta\in\mathcal{L}.
	\end{eqnarray*}
	
	Therefore we conclude that
	$$\left[\lambda,+\infty\right)\subseteq\mathcal{L}.$$
The proof is now complete.
\end{proof}

An interesting byproduct of the above proof is the following corollary.
\begin{corollary}\label{cor10}
	If hypotheses $H(\xi),H(\beta),H(f),H(g),H_0$ hold, $\lambda\in\mathcal{L},\vartheta>\lambda$, and $u_\lambda\in S(\lambda)\subseteq D_+$, then $\vartheta\in\mathcal{L}$ and we can find $u_\vartheta\in S(\vartheta)\subseteq D_+$ such that $u_\lambda\leq u_\vartheta,u_\vartheta\neq u_\lambda$.
\end{corollary}

In fact, we can improve  the conclusion of this corollary as follows.
\begin{prop}\label{prop11}
	Assume that $\lambda\in\mathcal{L}$, $\vartheta>\lambda$, and $u_\lambda\in S(\lambda) \subseteq D_+$. Then $\vartheta\in\mathcal{L}$ and we can find $u_\vartheta\in S(\vartheta)\subseteq D_+$ such that $u_\vartheta-u_\lambda\in {\rm int}\, C_+$.
\end{prop}
\begin{proof}
	From Corollary \ref{cor10} we already know that $\vartheta\in\mathcal{L}$ and that there exists $u_\vartheta\in S(\vartheta)\subseteq D_+$ such that
	$$u_\vartheta-u_\lambda\in C_+\backslash\{0\}.$$
	
	Let $\rho=||u_\vartheta||_\infty$ and $\hat{\xi}_\rho^\vartheta>0$ as in $H_0$. We can always assume that $\hat{\xi}_\rho^\vartheta>||\xi||_\infty$. We have
	\begin{eqnarray}\label{eq20}
		&&-\Delta_pu_\lambda+(\xi(z)+\hat{\xi}_\rho^\vartheta)u_\lambda^{p-1}\nonumber\\
		&=&f(z,u_\lambda,\lambda)+g(z,u_\lambda)+\hat{\xi}_\rho^\vartheta u_\lambda^{p-1}\nonumber\\
		&\leq&f(z,u_\vartheta,\lambda)+g(z,u_\vartheta)+\hat{\xi}_\rho^\vartheta u_\vartheta^{p-1}\ (\mbox{see hypothesis}\ H_0\ \mbox{and recall that}\ \lambda<\vartheta)\nonumber\\
		&=&f(z,u_\vartheta,\vartheta)+g(z,u_\vartheta)+\hat{\xi}_\rho^\vartheta u_\vartheta^{p-1}-[f(z,u_\vartheta,\vartheta)-f(z,u_\vartheta,\lambda)]\nonumber\\
		&\leq&f(z,u_\vartheta,\vartheta)+g(z,u_\vartheta)+\hat{\xi}_\rho^{\vartheta}u_\vartheta^{p-1}-\tilde{\eta}_s\nonumber\\
		&&\mbox{with}\ 0<s=\min\limits_{\overline{\Omega}}u_\vartheta\ (\mbox{recall that}\ u_\vartheta\in D_+\ \mbox{and see hypothesis}\ H(f)(iv))\nonumber\\
		&<&f(z,u_\vartheta,\vartheta)+g(z,u_\vartheta)+\hat{\xi}_\rho^\vartheta u_\vartheta^{p-1}\nonumber\\
		&=&-\Delta_pu_\vartheta+(\xi(z)+\hat{\xi}_\rho^\vartheta)u_\vartheta^{p-1}\ \mbox{for almost all}\ z\in\Omega\ (\mbox{since}\ u_\vartheta\in S(\vartheta)).
	\end{eqnarray}
	
	Since $\tilde{\eta}_s>0$, from (\ref{eq20}) and Proposition \ref{prop4}, we infer that
	$$u_\vartheta-u_\lambda\in {\rm int}\, C_+.$$
The proof is complete.
\end{proof}

Now let $\lambda>\lambda^*$. By Proposition \ref{prop9} we know that $\lambda\in\mathcal{L}$. We show that problem \eqref{eqp} has at least two positive solutions.
\begin{prop}\label{prop12}
	If  $\lambda>\lambda^*$, then problem \eqref{eqp} has at least two positive solutions
	$$u_0,\ \hat{u}\in D_+,\ u_0\neq\hat{u}.$$
\end{prop}
\begin{proof}
	As we have already mentioned, $\lambda\in\mathcal{L}$. Let $\lambda^*<\eta<\lambda<\vartheta$. We have $\eta,\vartheta\in\mathcal{L}$ (see Proposition \ref{prop9}). According to Proposition \ref{prop11}, there are $u_\vartheta\in S(\vartheta)\subseteq D_+$ and $u_\mu\in S(\mu)\subseteq D_+$ such that
	$$u_\vartheta-u_\mu\in {\rm int}\, C_+.$$
	
	We introduce the Carath\'eodory function $l_\lambda(z,x)$ defined by
	\begin{eqnarray}\label{eq21}
		&&l_\lambda(z,x)=\left\{\begin{array}{ll}
			f(z,u_\eta(z),\lambda)+g(z,u_\eta(z))+\mu u_\eta(z)^{p-1}&\mbox{if}\ x<u_\eta(z)\\
			f(z,x,\lambda)+g(z,x)+\mu x^{p-1}&\mbox{if}\ u_\eta(z)\leq x\leq u_\vartheta(z)\\
			f(z,u_\vartheta(z),\lambda)+g(z,x)+\mu u_\vartheta(z)^{p-1}&\mbox{if}\ u_\vartheta(z)<x.
		\end{array}\right.
	\end{eqnarray}
	
	Recall that $\mu>||\xi||_\infty$. We set $L_\lambda(z,x)=\int^x_0l_\lambda(z,s)ds$ and consider the $C^1$-functional $\hat{\varphi}_\lambda:W^{1,p}(\Omega)\rightarrow\RR$ defined by
	$$\hat{\varphi}_\lambda(u)=\frac{1}{p}\gamma_p(u)+\frac{\mu}{p}||u||^p_p-\int_\Omega L_\lambda(z,u)dz\ \mbox{for all}\ u\in W^{1,p}(\Omega).$$
	
	Since $\mu>||\xi||_\infty$, it is clear from (\ref{eq21}) that $\hat{\varphi}_\lambda(\cdot)$ is coercive. Also, it is sequentially weakly lower semicontinuous. So, we can find $u_0\in W^{1,p}(\Omega)$ such that
	\begin{eqnarray}\label{eq22}
		&&\hat{\varphi}_\lambda(u_0)=\inf\left\{\hat{\varphi}_\lambda(u):u\in W^{1,p}(\Omega)\right\},\nonumber\\
		&\Rightarrow&\hat{\varphi}'_\lambda(u_0)=0,\nonumber\\
		&\Rightarrow&\left\langle A(u_0),h\right\rangle+\int_\Omega(\xi(z)+\mu)|u_0|^{p-2}u_0hdz+\int_{\partial\Omega}\beta(z)|u_0|^{p-2}u_0hd\sigma=\nonumber\\
		&&\int_\Omega l_\lambda(z,u_0)hdz\ \mbox{for all}\ h\in W^{1,p}(\Omega).
	\end{eqnarray}
	
	In (\ref{eq22}) we first choose $h=(u_0-u_\vartheta)^+\in W^{1,p}(\Omega)$. Then
	\begin{eqnarray*}
		&&\left\langle A(u_0),(u_0-u_\vartheta)^+\right\rangle+\int_\Omega(\xi(z)+\mu)u_0^{p-1}(u_0-u_\vartheta)^+dz+
\int_{\partial\Omega}\beta(z)u_0^{p-1}(u_0-u_\vartheta)^+d\sigma\\
		&=&\int_\Omega[f(z,u_\vartheta,\lambda)+g(z,u_\vartheta)+\mu u_\vartheta^{p-1}](u_0-u_\vartheta)^+dz\ (\mbox{see (\ref{eq21})})\\
		&\leq&\int_\Omega[f(z,u_\vartheta,\vartheta)+g(z,u_\vartheta)+\mu u_\vartheta^{p-1}](u_0-u_\vartheta)^+dz\\
		&&(\mbox{see hypothesis}\ H(f)(iv)\ \mbox{and recall that}\ \lambda<\vartheta)\\
		&=&\left\langle A(u_\vartheta),(u_0-u_\vartheta)^+\right\rangle+\int_\Omega(\xi(z)+\mu)u_\vartheta^{p-1}(u_0-u_\vartheta)^+dz+\int_{\partial\Omega}\beta(z)u_\vartheta^{p-1}(u_0-u_\vartheta)^+d\sigma\\
		&&(\mbox{since}\ u_\vartheta\in S(\vartheta)),\\
		&\Rightarrow&u_0\leq u_\vartheta\ (\mbox{see Proposition \ref{prop2} and recall that}\ \mu>||\xi||_\infty).
	\end{eqnarray*}
	
	Similarly, if in (\ref{eq22}) we choose $h=(u_\eta-u_0)^+\in W^{1,p}(\Omega)$, we can show that
	$$u_\eta\leq u_0.$$
	
	So, we have proved that
	\begin{equation}\label{eq23}
		u_0\in[u_\eta,u_\vartheta].
	\end{equation}
	
	Then  it follows from (\ref{eq21}), (\ref{eq22}) and (\ref{eq23}) that $u_0\in S(\lambda)\subseteq D_+$. Moreover, arguing as in the proof of Proposition \ref{prop11}, via Proposition \ref{prop4}, we show that
	\begin{eqnarray}\label{eq24}
		&&u_\vartheta-u_0\in {\rm int}\, C_+\ \mbox{and}\ u_0-u_\eta\in {\rm int}\, C_+,\nonumber\\
		&\Rightarrow&u_0\in {\rm int}_{C^1(\overline{\Omega})}[u_\eta,u_\vartheta].
	\end{eqnarray}
	
	Let $\psi_\lambda:W^{1,p}(\Omega)\rightarrow\RR$ be the $C^1$-functional introduced in the proof of Proposition \ref{prop8}. From (\ref{eq21}) it is clear that
	\begin{equation}\label{eq25}
		\left.\psi_\lambda\right|_{[u_\eta,u_\vartheta]}=\left.\hat{\varphi}_\lambda\right|_{[u_\eta,u_\vartheta]}+\hat{k}_\lambda\ \mbox{with}\ \hat{k}_\lambda\in\RR.
	\end{equation}
	
	From (\ref{eq24}) and (\ref{eq25}) it follows that
	\begin{eqnarray}\label{eq26}
		&&u_0\ \mbox{is local}\ C^1(\overline{\Omega})-\mbox{minimizer of}\ \psi_\lambda,\nonumber\\
		&\Rightarrow&u_0\ \mbox{is local}\ W^{1,p}(\Omega)-\mbox{minimizer of}\ \psi_\lambda\ (\mbox{see Proposition \ref{prop3}}).
	\end{eqnarray}
	
	Hypotheses $H(f)(ii)$ and $H(g)(ii)$ imply that given $\epsilon>0$, we can find $\delta>0$ such that
	\begin{equation}\label{eq27}
		F(z,x,\lambda)\leq\frac{\epsilon}{p}x^p,\ G(z,x)\leq\frac{1}{p}(\eta_0(z)+\epsilon)x^p\ \mbox{for almost all}\ z\in\Omega\ \mbox{and all}\ 0\leq x\leq\delta.
	\end{equation}
	
	For all $u\in C^1(\overline{\Omega})$ with $||u||_{C^1(\overline{\Omega})}\leq\delta$, we have
	\begin{eqnarray*}
		&&\psi_\lambda (u)\geq\frac{1}{p}\gamma_p(u^-)+\frac{\mu}{p}||u^-||^p_p+\frac{1}{p}\gamma_p(u^+)-\frac{1}{p}\int_{\Omega}\eta_0(z)(u^+)^pdz-\frac{2\epsilon}{p}||u^+||^p_p\\
		&&(\mbox{see (\ref{eq27}) and recall the definition of}\ \psi_\lambda\ \mbox{in the proof of Proposition \ref{prop8}})\\
		&&\geq c_9||u^-||^p+\frac{1}{p}(c_0-2\epsilon)\,||u^+||^p\ \mbox{for some}\ c_9>0\\
		&&(\mbox{recall that}\ \mu>||\xi||_\infty\ \mbox{and use Lemma \ref{lem6}}).
	\end{eqnarray*}
	
	Choosing $\epsilon\in\left(0,\frac{c_0}{2}\right)$, we conclude that
	\begin{eqnarray}\label{eq28}
		&&\psi_\lambda(u)\geq c_{10}||u||^p\ \mbox{for some}\ c_{10}>0,\ \mbox{all}\ u\in C^1(\overline{\Omega})\ \mbox{with}\ ||u||_{C^1(\overline{\Omega})}\leq\delta,\nonumber\\
		&\Rightarrow&u=0\ \mbox{is a local}\ C^1(\overline{\Omega})-\mbox{minimizer of}\ \psi_\lambda,\nonumber\\
		&\Rightarrow&u=0\ \mbox{is a local}\ W^{1,p}(\Omega)-\mbox{minimizer of}\ \psi_\lambda\ (\mbox{see Proposition \ref{prop3}}).
	\end{eqnarray}
	
	Without any loss of generality, we may assume that
	$$0=\psi_\lambda(0)\leq\psi_\lambda(u_0).$$
	
	The analysis is similar if the opposite inequality holds using (\ref{eq28}) instead of (\ref{eq26}). In addition, we may assume that $K_{\psi_\lambda}$ is finite. Otherwise since $K_{\psi_\lambda}\subseteq D_+\cup\{0\}$, we see that we already have an infinity of positive solutions for problem \eqref{eqp} and so we are done. Then on account of (\ref{eq26}), we can find $\rho\in(0,1)$ small such that
	\begin{equation}\label{eq29}
		0=\psi_\lambda(0)\leq\psi_\lambda(u_0)<\inf\{\psi_\lambda(u):||u-u_0||=\rho\}=m_\lambda,\ ||u_0||>\rho
	\end{equation}
	(see Aizicovici, Papageorgiou \& Staicu \cite{1}, proof of Proposition 29).
	
	From the proof of Proposition \ref{prop8} we know that
	\begin{eqnarray}\label{eq30}
		&&\psi_\lambda(\cdot)\ \mbox{is coercive},\nonumber\\
		&\Rightarrow&\psi_\lambda(\cdot)\ \mbox{satisfies the PS-condition (see Section 2)}.
	\end{eqnarray}
	
	From (\ref{eq29}) and (\ref{eq30}) it follows that we can use Theorem \ref{th1} (the mountain pass theorem). So, we can find $\hat{u}\in W^{1,p}(\Omega)$ such that
	\begin{eqnarray*}
		&&\hat{u}\in K_{\psi_\lambda}\subseteq D_+\cup\{0\}\ \mbox{and}\ 0<m_\lambda\leq\psi_\lambda(\hat{u}),\\
		&\Rightarrow&\hat{u}\in S(\lambda) \subseteq D_+\ \mbox{and}\ \hat{u}\neq u_0\ (\mbox{see (\ref{eq29})}).
	\end{eqnarray*}
The proof is now complete.
\end{proof}

Next, we show that the critical parameter value $\lambda^*>0$ is also admissible  (that is, $\lambda^*\in\mathcal{L}$).
\begin{prop}\label{prop13}
	We have that $\lambda^*\in\mathcal{L}$.
\end{prop}
\begin{proof}
	Let $\{\lambda_n\}_{n\geq 1}\subseteq(\lambda^*,+\infty)$ be such that $\lambda_n\rightarrow(\lambda^*)^+$ as $n\rightarrow\infty$. From the proof of Proposition \ref{prop11}, we know that we can find $u_n\in S(\lambda_n)\subseteq D_+$ ($n\in\NN$) decreasing. We have
	\begin{eqnarray}
		&&0\leq u_n\leq u_1\ \mbox{for all}\ n\in\NN,\label{eq31}\\
		&&\left\langle A(u_n),h\right\rangle+\int_\Omega \xi(z)u_n^{p-1}hdz+\int_{\partial\Omega}\beta(z)u_n^{p-1}hd\sigma=\int_{\Omega}[f(z,u_n,\lambda_n)+g(z,u_n)]hdz\nonumber\\
		&&\mbox{for all}\ h\in W^{1,p}(\Omega)\ \mbox{and all}\ n\in\NN\label{eq32}.
	\end{eqnarray}
	
	In (\ref{eq32}) we choose $h=u_n\in W^{1,p}(\Omega)$. Using (\ref{eq31}) and hypotheses $H(\xi),H(\beta),H(f)(i)$, $H(g)(i)$, we see that
	$$\{u_n\}_{n\geq 1}\subseteq W^{1,p}(\Omega)\ \mbox{is bounded}.$$
	
	Therefore, by passing to a subsequence if necessary, we may assume that
	\begin{equation}\label{eq33}
		u_n\stackrel{w}{\rightarrow}u_{\lambda^*}\ \mbox{in}\ W^{1,p}(\Omega)\ \mbox{and}\ u_n\rightarrow u_{\lambda^*}\ \mbox{in}\ L^p(\Omega)\ \mbox{and}\ L^p(\partial\Omega).
	\end{equation}
	
	For every $n\in\NN$, we have
	\begin{eqnarray}\label{eq34}
		&&-\Delta_pu_n(z)+\xi(z)u_n(z)^{p-1}=f(z,u_n(z),\lambda_n)+g(z,u_n(z))\ \mbox{for almost all}\ z\in\Omega,\nonumber\\
		&&\frac{\partial u}{\partial n_p}+\beta(z)u_n^{p-1}=0\ \mbox{on}\ \partial\Omega\ (\mbox{see Papageorgiou \& R\u{a}dulescu \cite{15}}).
	\end{eqnarray}
	
	From Papageorgiou \& R\u{a}dulescu \cite[Proposition 7]{16}  and (\ref{eq34}), we know that we can find $c_{11}>0$ such that
	$$||u_n||_\infty\leq c_{11}\ \mbox{for all}\ n\in\NN.$$
	
	Then invoking Theorem 2 of Lieberman \cite{12}, we can find $\gamma\in(0,1)$ and $c_{12}>0$ such that
	\begin{equation}\label{eq35}
		u_n\in C^{1,\gamma}(\overline{\Omega})\ \mbox{and}\ ||u_n||_{C^{1,\gamma}(\overline{\Omega})}\leq c_{12}\ \mbox{for all}\ n\in\NN.
	\end{equation}
	
	Since $C^{1,\gamma}(\overline{\Omega})$ is compactly embedded in $C^1(\overline{\Omega})$, we have from (\ref{eq33}) and (\ref{eq35})
	\begin{equation}\label{eq36}
		u_n\rightarrow u_{\lambda^*}\ \mbox{in}\ C^1(\overline{\Omega}).
	\end{equation}
	
	Passing to the limit as $n\rightarrow\infty$ in (\ref{eq32}) and using (\ref{eq36}), we obtain
	\begin{eqnarray}\label{eq37}
		&&\left\langle A(u_{\lambda^*}),h\right\rangle+\int_\Omega\xi(z)u_{\lambda^*}^{p-1}hdz+\int_{\partial\Omega}\beta(z)u_{\lambda^*}^{p-1}hd\sigma=\nonumber\\
		&&\int_\Omega[f(z,u_{\lambda^*},\lambda^*)+g(z,u_{\lambda^*})]hdz\ \mbox{for all}\ h\in W^{1,p}(\Omega),\\
		&\Rightarrow&u_{\lambda^*}\ \mbox{is a nonnegative solution of}\ (P_{\lambda^*}).\nonumber
	\end{eqnarray}
	
	We need to show that $u_{\lambda^*}\neq 0$. Then we will have $u_{\lambda^*}\in S(\lambda^*)\subseteq D_+$ and $\lambda^*\in\mathcal{L}$.
	
	Arguing by contradiction, suppose that $u_{\lambda^*}=0$. Then we have from (\ref{eq36}) 
	\begin{equation}\label{eq38}
		u_n\rightarrow 0\ \mbox{in}\ C^1(\overline{\Omega}).
	\end{equation}
	
	Hypotheses $H(f)(ii)$ and $H(g)(ii)$ imply that given $\epsilon>0$, we can find $\delta=\delta(\epsilon)>0$ such that
	\begin{equation}\label{eq39}
		f(z,x,\lambda_1)x\leq\epsilon x^p,\ g(z,x)x\leq(\eta_0(z)+\epsilon)x^p\ \mbox{for almost all}\ z\in\Omega,\ \mbox{all}\ 0\leq x\leq \delta.
	\end{equation}
	
	In (\ref{eq37}) we choose $h=u_n\in W^{1,p}(\Omega)$. Then
	\begin{eqnarray}\label{eq40}
		\gamma_p(u_n)&=&\int_\Omega[f(z,u_n,\lambda_n)+g(z,u_n)]u_ndz\nonumber\\
		&\leq&\int_{\Omega}[f(z,u_n,\lambda_1)+g(z,u_n)]u_ndz\ \mbox{for all}\ n\in\NN\ (\mbox{see hypothesis}\ H(f)(iv)).
	\end{eqnarray}
	
	From (\ref{eq38}), we see that we can find $n_0\in\NN$ such that
	\begin{equation}\label{eq41}
		u_ n(z)\in\left(0,\delta\right]\ \mbox{for all}\ z\in\overline{\Omega},\ n\geq n_0.
	\end{equation}
	
	Then from (\ref{eq39}), (\ref{eq40}), (\ref{eq41}), we see that
	\begin{eqnarray*}
		&&\gamma_p(u_n)-\int_\Omega\eta_0(z)u_n^pdz\leq 2\epsilon||u_n||^p_p\ \mbox{for all}\ n\geq n_0,\\
		&\Rightarrow&c_0||u_n||^p\leq 2\epsilon||u_n||^p_p\ \mbox{for all}\ n\geq n_0\ (\mbox{see Lemma \ref{lem6}}),\\
		&\Rightarrow&c_0\leq 2\epsilon.
	\end{eqnarray*}
	
	Since $\epsilon>0$ is arbitrary, choosing $\epsilon\in\left(0,\frac{c_0}{2}\right)$, we get a contradiction. Therefore $u_{\lambda^*}\neq 0$ and so $u_{\lambda^*}\in S(\lambda^*)\subseteq D_+$, hence $\lambda^*\in\mathcal{L}$.
\end{proof}

So, we conclude that
$$\mathcal{L}=\left[\lambda^*,+\infty\right).$$

\section{Minimal positive solutions}

In this section we show that for every $\lambda\in\mathcal{L}$, problem \eqref{eqp} has a smallest positive solution $\bar{u}_\lambda\in D_+$ and we study the monotonicity and continuity properties of the map $\lambda\mapsto\bar{u}_\lambda.$

From Papageorgiou, R\u{a}dulescu \& Repov\v{s} \cite{18} (see the proof of Proposition 7), we know that $S(\lambda)$ is downward directed, that is, if $u_1,u_2\in S(\lambda)$, then we can find $u\in S(\lambda)$ such that $u\leq u_1,u\leq u_2$.
\begin{prop}\label{prop15}
	Assume that $\lambda\in\mathcal{L}=\left[\lambda^*,+\infty\right)$. Then problem \eqref{eqp} admits a smallest positive solution $\bar{u}_\lambda\in S(\lambda)\subseteq D_+$ (that is, $\bar{u}_\lambda\leq u$ for all $u\in S(\lambda)$).
\end{prop}
\begin{proof}
	According to Lemma 3.10 of Hu \& Papageorgiou \cite[p. 178]{11} and since $S(\lambda)$ is downward directed, we can find a decreasing sequence $\{u_n\}_{n\geq 1}\subseteq S(\lambda)$ such that
	$$\inf S(\lambda)=\inf\limits_{n\geq 1}u_n.$$
	
	We have
	\begin{eqnarray}
		&&0\leq u_n\leq u_1\ \mbox{for all}\ n\in\NN,\label{eq42}\\
		&&\left\langle A(u_n),h\right\rangle+\int_\Omega\xi(z)u_n^{p-1}hdz+\int_{\partial\Omega}\beta(z)u_n^{p-1}hd\sigma=\nonumber\\
		&&\int_\Omega[f(z,u_n.\lambda)+g(z,u_n)]hdz\ \mbox{for all}\ h\in W^{1,p}(\Omega)\ \mbox{and all}\ n\in\NN\label{eq43}.
	\end{eqnarray}
	
	By reasoning as in the proof of Proposition \ref{prop13} (see the part of the proof after (\ref{eq32})) and using (\ref{eq42}) and (\ref{eq43}), we obtain
	\begin{eqnarray*}
		&&u_n\rightarrow \bar{u}_\lambda\ \mbox{in}\ C^1(\overline{\Omega})\ \mbox{with}\ \bar{u}_\lambda\in S(\lambda),\\
		&\Rightarrow&\bar{u}_\lambda=\inf S(\lambda).
	\end{eqnarray*}
The proof is now complete.
\end{proof}

\begin{prop}\label{prop16}
	The map $\lambda\mapsto\bar{u}_\lambda$ from $\overset{o}{\mathcal{L}}=(\lambda^*,+\infty)$ into $C^1(\overline{\Omega})$ has the following properties:
	\begin{itemize}
		\item	it is strictly monotone in the sense that
		$$\overset{o}{\mathcal{L}}\ni\lambda<\vartheta\Rightarrow\bar{u}_\vartheta-\bar{u}_\lambda\in {\rm int}\, C_+;$$
		\item it is left continuous.
	\end{itemize}
\end{prop}
\begin{proof}
	First, we show the strict monotonicity of the map $\lambda\mapsto\bar{u}_\lambda$. So, let $\lambda\in\overset{o}{\mathcal{L}}$ and $\vartheta >\lambda$. Then $\vartheta\in\mathcal{L}$ and let $\bar{u}_\vartheta\in S(\vartheta)\subseteq D_+$ be the minimal solution of problem ($P_\vartheta$). From the proof of Proposition \ref{prop12}, we know that we can find $u_\lambda\in S(\lambda)\subseteq D_+$ such that
	\begin{eqnarray*}
		&&\bar{u}_\vartheta-u_\lambda\in {\rm int}\, C_+\ (\mbox{see (\ref{eq24})}),\\
		&\Rightarrow&\bar{u}_\vartheta-\bar{u}_\lambda\in {\rm int}\, C_+\ (\mbox{since}\ \bar{u}_\lambda\leq u_\lambda).
	\end{eqnarray*}
	
	This proves the strict monotonicity of the map $\lambda\mapsto\bar{u}_\lambda$ from $\overset{o}{\mathcal{L}}=(\lambda^*,+\infty)$ into $C^1(\overline{\Omega})$.
	
	Next, we show the left continuity of the map $\lambda\mapsto\bar{u}_\lambda$. So, let $\{\lambda_n\}_{n\geq 1}\subseteq\overset{o}{\mathcal{L}}$ and assume that $\lambda_n\rightarrow\lambda^-$. From the first part of the proof, we have
	$$0\leq\bar{u}_{\lambda_n}\leq\bar{u}_\lambda\ \mbox{for all}\ n\geq 1$$
	
	Then as before (see the proof of Proposition \ref{prop13}), we can say that
	\begin{equation}\label{eq44}
		\bar{u}_{\lambda_n}\rightarrow\tilde{u}_\lambda\ \mbox{in}\ C^1(\overline{\Omega})\ \mbox{as}\ n\rightarrow\infty\end{equation}
		and
$$\tilde{u}_\lambda\in S(\lambda)\subseteq D_+.$$
	
	Suppose that $\tilde{u}_\lambda\neq\bar{u}_\lambda$. Then we can find $z_0\in\overline{\Omega}$ such that
	\begin{eqnarray*}
		&&\bar{u}_\lambda(z_0)<\tilde{u}_\lambda(z_0),\\
		&\Rightarrow&\bar{u}_\lambda(z_0)<\bar{u}_{\lambda_n}(z_0)\ \mbox{for all}\ n\geq n_0,
	\end{eqnarray*}
	which contradicts the first part of the proposition. Therefore
	\begin{eqnarray*}
		&&\tilde{u}_\lambda=\bar{u}_\lambda,\\
		&\Rightarrow&\lambda\mapsto\bar{u}_\lambda\ \mbox{is continuous from}\ \overset{o}{\mathcal{L}}\ \mbox{into}\ C^1(\overline{\Omega}).
	\end{eqnarray*}
The proof is now complete.
\end{proof}

\begin{remark}
	In our setting the equation was nonuniformly nonresonant as $x\rightarrow+\infty$ (see hypotheses $H(f)(ii),H(g)(ii)$). Is it possible also to treat the resonant case, that is,
	$$\limsup\limits_{x\rightarrow+\infty}\frac{g(z,x)}{x^{p-1}}\leq\hat{\lambda}_1\ \mbox{uniformly for almost all}\ z\in\Omega.$$

However, it is unknown what happens for the asymptotic behavior as $x\rightarrow+\infty$, provided that we are nonresonant with respect to $\hat{\lambda}_1$, but from above the principal eigenvalue, that is,
$$\liminf\limits_{x\rightarrow+\infty}\frac{g(z,x)}{x^{p-1}}\geq\hat{\eta}>\hat{\lambda}_1\ \mbox{uniformly for almost all}\ z\in\Omega.$$

A careful inspection of the arguments of this paper reveals that for the nonresonant case, but from above $\hat{\lambda}_1$, if a bifurcation-type result holds, then it occurs for small values of $\lambda>0$. This also suggests that if we want to extend the results of this paper to the resonant case, we must have resonance from the left of $\hat{\lambda}_1$, in the sense that
$$\hat{\lambda}_1x^{p-1}-[f(z,x,\lambda)+g(z,x)]\rightarrow+\infty\ \mbox{uniformly for almost all}\ z\in\Omega,\ \mbox{as}\ x\rightarrow+\infty.$$

In this way we can preserve the coercivity of the energy functional and we hope to be able to extend the results of paper to the resonant case.
\end{remark}

\medskip
{\bf Acknowledgements.} The authors wish to thank a very knowledgeable referee for his/her corrections and remarks, which have improved our presentation. This research was supported by the Slovenian Research Agency grants
P1-0292, J1-8131, J1-7025, N1-0064, and N1-0083. V.D.~R\u adulescu acknowledges the support through a grant of the Romanian Ministry of Research and Innovation, CNCS--UEFISCDI, project number PN-III-P4-ID-PCE-2016-0130,
within PNCDI III.

\end{document}